\documentclass[portuges,12pt,letter]{article}
\usepackage[centertags]{amsmath}
\usepackage{amsfonts}
\usepackage{newlfont}
\usepackage{amscd}
\usepackage{graphics}
\usepackage{epsfig}
\usepackage{indentfirst}
\usepackage{amsxtra}
\usepackage[latin1]{inputenc}
\usepackage{amssymb, amsmath}
\usepackage{amsthm}
\usepackage[mathscr]{eucal}

\newtheorem{thm}{Theorem}[section]
\newtheorem{cor}[thm]{Corolary}

\newtheorem{dfn}[thm]{Definition}
\newtheorem{obe}[thm]{Remark}

\author{Fabio Silva Botelho \\ Department of Mathematics \\  Federal University of Santa Catarina, UFSC \\
Florian\'{o}polis, SC - Brazil}

\title{\bf  A note on the Korn inequality in a n-dimensional context}
\date{}
\begin{document}
\maketitle

\abstract{In this short communication, we present a new proof for the Korn inequality in a n-dimensional context.
The results are based on standard tools of real and functional analysis. For the final result the standard Poincar\'{e} inequality plays a fundamental role.}

\section{Introduction}
In this article we present a proof for the Korn inequality in $\mathbb{R}^n$. The results are based on standard tools of functional analysis and Sobolev spaces theory.

We highlight such a proof is relatively simple and easy to follow since is established in a very transparent and clear fashion.

\begin{obe} Generically throughout the text we denote
$$\|u\|_{0,2,\Omega}=\left(\int_\Omega |u|^2\;dx\right)^{1/2},\; \forall u \in L^2(\Omega),$$
and
$$\|u\|_{0,2,\Omega}=\left( \sum_{j=1}^n \|u_j\|_{0,2,\Omega}^2\right)^{1/2},\; \forall u=(u_1,\ldots,u_n) \in L^2(\Omega;\mathbb{R}^n).$$
Moreover,
$$\|u\|_{1,2,\Omega}=\left( \|u\|_{0,2,\Omega}^2+\sum_{j=1}^n \|u_{x_j}\|_{0,2,\Omega}^2\right)^{1/2},\; \forall u \in W^{1,2}(\Omega),$$
where we shall also refer throughout the text to the  well known corresponding analogous norm for $u \in W^{1,2}(\Omega;\mathbb{R}^n).$
\end{obe}

At this point, we first introduce the following definition.
\begin{dfn} Let $\Omega \subset \mathbb{R}^n$ be an open, bounded set. We say that $\partial \Omega$ is $\hat{C}^1$ if such a manifold is oriented and for each  $x_0 \in \partial \Omega$, denoting $\hat{x}=(x_1,...,x_{n-1})$ for a local coordinate system compatible with the manifold $\partial \Omega$ orientation,
there exist $r>0$ and a function $f(x_1,...,x_{n-1})=f(\hat{x})$ such that
$$W=\overline{\Omega} \cap B_r(x_0)=\{x \in B_r(x_0)\;|\; x_n \leq f(x_1,...,x_{n-1})\}.$$ Moreover $f(\hat{x})$ is  a Lipschitz continuous function,  so that
$$|f(\hat{x})-f(\hat{y})| \leq C_1 |\hat{x}-\hat{y}|_2, \text{ on its domain},$$ for some $C_1>0$. Finally, we assume
$$\left\{\frac{\partial f(\hat{x})}{\partial x_k}\right\}_{k=1}^{n-1}$$ is classically defined,  almost everywhere also on its concerning domain, so that $f \in W^{1,2}$.
\end{dfn}
At this point we recall the following result found in \cite{512}, at page 222 in its Chapter 11.
\begin{thm}
\label{ppp.18} Assume $\Omega \subset \mathbb{R}^n$ is an open bounded set, and that
$\partial \Omega$ is $\hat{C}^1$. Let $1\leq p < \infty,$ and let $V$ be a bounded open set such that $\Omega\subset\subset V$.
Then there exists a bounded linear operator $$E: W^{1,p}(\Omega) \rightarrow W^{1,p}(\mathbb{R}^n),$$
such that
for each $ u \in W^{1,p}(\Omega)$ we have:
\begin{enumerate}
\item $Eu=u, \text{ a.e. in } \Omega,$
\item $Eu $ has support in $V$,
and
\item $\|Eu\|_{1,p,\mathbb{R}^n}\leq C \|u\|_{1,p,\Omega},$
where the constant depend only on $p,\Omega, \text{ and } V.$
\end{enumerate}
\end{thm}
\begin{obe} Considering the proof of such a result, the constant $C>0$ may be also such that
$$\|E(e_{ij}(u))\|_{0,2,V} \leq C \|e_{ij}(u)\|_{0,2,\Omega},\; \forall u \in W^{1,2}(\Omega;\mathbb{R}^n),
\; \forall i,j \in \{1,\ldots,n\},$$ for the operator $e:W^{1,2}(\Omega;\mathbb{R}^n) \rightarrow L^2(\Omega;\mathbb{R}^{n \times n})$ specified in the next theorem.
\end{obe}

Finally, as the meaning is clear, we may simply denote $Eu=u.$

\section{ The main results, the Korn inequalities}

Our main result is summarized by the following theorem.

\begin{thm} Let $\Omega \subset \mathbb{R}^n$ be an open, bounded and connected set with a $\hat{C}^1$ (Lipschitzian) boundary $\partial \Omega$.

Define $e:W^{1,2}(\Omega;\mathbb{R}^n) \rightarrow L^2(\Omega;\mathbb{R}^{n \times n})$ by
$$e(u)=\{e_{ij}(u)\}$$ where $$e_{ij}(u)=\frac{1}{2}( u_{i,j}+u_{j,i}),\;\forall i,j \in \{1,\ldots,n\}.$$

Define also,
$$\|e(u)\|_{0,2,\Omega}=\left(\sum_{i=1}^n\sum_{j=1}^n\|e_{ij(u)}\|_{0,2,\Omega}^2\right)^{1/2}.$$

Let $L \in \mathbb{R}^+$ be such $V=[-L,L]^n$ is also such that $\overline{\Omega} \subset V^0.$

Under such hypotheses, there exists $C(\Omega,L) \in \mathbb{R}^+$ such that
$$\|u\|_{1,2,\Omega} \leq C(\Omega,L)\left( \|u\|_{0,2,\Omega}+\|e(u)\|_{0,2,\Omega}\right),\;\forall u \in W^{1,2}(\Omega;\mathbb{R}^n).$$
\end{thm}
\begin{proof}
Suppose, to obtain contradiction, the concerning claim does not hold.

Hence, for each $k \in \mathbb{N}$ there exists $u_k \in W^{1,2}(\Omega;\mathbb{R}^n)$ such that
$$\|u_k\|_{1,2,\Omega} > k\left( \|u_k\|_{0,2,\Omega}+\|e(u_k)\|_{0,2,\Omega}\right).$$

In particular defining $$v_k=\frac{u_k}{\|u_k\|_{1,2,\Omega}}$$ we obtain
$$\|v_k\|_{1,2,\Omega}=1 >  k\left( \|v_k\|_{0,2,\Omega}+\|e(v_k)\|_{0,2,\Omega}\right),$$
so that
$$\left( \|v_k\|_{0,2,\Omega}+\|e(v_k)\|_{0,2,\Omega}\right)< \frac{1}{k},\; \forall k \in \mathbb{N}.$$
From this we have got,
$$\|v_k\|_{0,2,\Omega}< \frac{1}{k},$$
and
$$\|e_{ij}(v_k)\|_{0,2,\Omega}< \frac{1}{k},\; \forall k \in \mathbb{N},$$
so that
$$\|v_k\|_{0,2,\Omega} \rightarrow 0, \text{ as } k \rightarrow \infty,$$
and
$$\|e_{ij}(v_k)\|_{0,2,\Omega} \rightarrow 0, \text{ as } k \rightarrow \infty.$$

In particular $$\|(v_k)_{j,j}\|_{0,2,\Omega} \rightarrow 0,\; \forall j \in \{1,\ldots,n\}.$$

At this point we recall the following identity in the distributional sense, found in \cite{30}, page 12,
\begin{equation}\label{us309.1}\partial_j(\partial_l v_i)=\partial_j e_{il}(v)+\partial_l e_{ij}(v)-\partial_i e_{jl}(v), \; \forall i,j,l \in \{1,\ldots,n\}.\end{equation}

Fix $j \in \{1,\ldots,n\}$ and observe that $$\|(v_k)_j\|_{1,2,V} \leq C \|(v_k)_j\|_{1,2,\Omega}$$ so that
$$\frac{C}{\|(v_k)_j\|_{1,2,V}} \geq \frac{1}{\;\|(v_k)_j\|_{1,2,\Omega}},\; \forall k \in \mathbb{N}.$$

Hence,
\begin{eqnarray}
&&\|(v_k)_j\|_{1,2,\Omega} \nonumber \\ &=&\sup_{\varphi \in C^1(\Omega)}\left\{\langle \nabla (v_k)_j,\nabla \varphi \rangle_{L^2(\Omega)}+ \langle  (v_k)_j, \varphi \rangle_{L^2(\Omega)}\;:\; \|\varphi\|_{1,2,\Omega} \leq 1\right\}
\nonumber \\ &=& \left\langle \nabla (v_k)_j,\nabla \left( \frac{(v_k)_j}{\|(v_k)_j\|_{1,2,\Omega}}\right) \right\rangle_{L^2(\Omega)} \nonumber \\ &&+ \left\langle  (v_k)_j, \left(\frac{(v_k)_j}{\|(v_k)_j\|_{1,2,\Omega}}\right) \right\rangle_{L^2(\Omega)}
\nonumber \\ &\leq&  C\left( \left\langle \nabla (v_k)_j,\nabla \left(\frac{(v_k)_j}{\|(v_k)_j\|_{1,2,V} }\right) \right\rangle_{L^2(V)} + \left\langle  (v_k)_j, \left(\frac{(v_k)_j}{\|(v_k)_j\|_{1,2,V}}\right) \right\rangle_{L^2(V)}\right)
\nonumber \\ &=& C \sup_{\varphi \in C^1_c(V)}\left\{ \langle \nabla (v_k)_j,\nabla \varphi \rangle_{L^2(V)}+ \langle  (v_k)_j, \varphi \rangle_{L^2(V)}\;:\; \|\varphi\|_{1,2,V} \leq 1\right\}.
\end{eqnarray}

Here, we recall that $C>0$ is the constant concerning the Extension Theorem \ref{ppp.18}.
From such results and (\ref{us309.1}), we have that
\begin{eqnarray}
&&\sup_{\varphi \in C^1(\Omega)}\left\{\langle \nabla (v_k)_j,\nabla \varphi \rangle_{L^2(\Omega)}+ \langle  (v_k)_j, \varphi \rangle_{L^2(\Omega)}\;:\; \|\varphi\|_{1,2,\Omega} \leq 1\right\}
\nonumber \\ &\leq& C\sup_{\varphi \in C^1_c(V)}\left\{ \langle \nabla (v_k)_j,\nabla \varphi \rangle_{L^2(V)}+ \langle  (v_k)_j, \varphi \rangle_{L^2(V)}\;:\; \|\varphi\|_{1,2,V} \leq 1\right\}
\nonumber \\ &=& C  \sup_{ \varphi \in C^1_c(V)}\left\{
\langle e_{jl}(v_k), \varphi_{,l}\rangle_{L^2(V)}+\langle e_{jl}(v_k), \varphi_{,l}\rangle_{L^2(V)}\right.
\nonumber \\ &&\left.-\langle e_{ll}(v_k),\varphi_{,j}\rangle_{L^2(V)}+\langle (v_k)_j,\varphi\rangle_{L^2(V)},\;:\;
\|\varphi\|_{1,2,V} \leq 1\right\}.
\end{eqnarray}

Therefore
\begin{eqnarray}
&&\|(v_k)_j\|_{\left(W^{1,2}(\Omega)\right)} \nonumber \\ &=&
\sup_{\varphi \in C^1(\Omega)}\{\langle \nabla (v_k)_j,\nabla \varphi \rangle_{L^2(\Omega)}+ \langle (v_k)_j, \varphi \rangle_{L^2(\Omega)}\;:\; \|\varphi\|_{1,2,\Omega} \leq 1\}\nonumber \\ &\leq& C\left( \sum_{l=1}^n \left\{\|e_{jl}(v_k)\|_{0,2,V}+\|e_{ll}(v_k)\|_{0,2,V}\right\} +\|(v_k)_j\|_{0,2,V}\right)
\nonumber \\ &\leq& C^2\left( \sum_{l=1}^n \left\{\|e_{jl}(v_k)\|_{0,2,\Omega}+\|e_{ll}(v_k)\|_{0,2,\Omega}\right\} + \|(v_k)_j\|_{0,2,\Omega}\right)
\nonumber \\ &<& \frac{7C^2}{k}.
\end{eqnarray}

Summarizing,

$$\|(v_k)_j\|_{\left(W^{1,2}(\Omega)\right)} < \frac{7C^2}{k}, \; \forall k \in \mathbb{N}.$$

Thus, similarly we may obtain $$\|(v_k)_{i,j}\|_{0,2,\Omega} \rightarrow 0,\; \forall i,j \in \{1,\ldots,n\}.$$

From this we have got $$\|v_k\|_{1,2,\Omega} \rightarrow 0, \text{ as } k \rightarrow \infty,$$
which contradicts $$\|v_k\|_{1,2,\Omega}=1,\; \forall k \in \mathbb{N}.$$

The proof is complete.
\end{proof}

\begin{cor}Let $\Omega \subset \mathbb{R}^n$ be an open, bounded and connected set with a $\hat{C}^1$ boundary $\partial \Omega$.
Define $e:W^{1,2}(\Omega;\mathbb{R}^n) \rightarrow L^2(\Omega;\mathbb{R}^{n \times n})$ by
$$e(u)=\{e_{ij}(u)\}$$ where $$e_{ij}(u)=\frac{1}{2}( u_{i,j}+u_{j,i}),\;\forall i,j \in \{1,\ldots,n\}.$$

Define also,
$$\|e(u)\|_{0,2,\Omega}=\left(\sum_{i=1}^n\sum_{j=1}^n\|e_{ij(u)}\|_{0,2,\Omega}^2 \right)^{1/2}.$$

Let $L \in \mathbb{R}^+$ be such $V=[-L,L]^n$ is also such that $\overline{\Omega} \subset V^0.$

Moreover, define $$\hat{H}_0=\{ u \in W^{1,2}(\Omega;\mathbb{R}^n)\::\; u=\mathbf{0}, \text{ on } \Gamma_0\},$$
where $\Gamma_0 \subset \partial \Omega$ is a measurable set such that the Lebesgue measure $m_{\mathbb{R}^{n-1}}(\Gamma_0)>0.$

Under such hypotheses, there exists $C(\Omega,L) \in \mathbb{R}^+$ such that
$$\|u\|_{1,2,\Omega} \leq C(\Omega,L)\; \|e(u)\|_{0,2,\Omega},\;\forall u \in \hat{H}_0.$$
\end{cor}
\begin{proof}
Suppose, to obtain contradiction, the concerning claim does not hold.

Hence, for each $k \in \mathbb{N}$ there exists $u_k \in \hat{H}_0$ such that
$$\|u_k\|_{1,2,\Omega} > k\;\|e(u_k)\|_{0,2,\Omega}.$$

In particular defining $$v_k=\frac{u_k}{\|u_k\|_{1,2,\Omega}}$$ similarly to the proof of the last theorem, we may obtain $$\|(v_k)_{j,j}\|_{0,2,\Omega} \rightarrow 0, \text{ as } k \rightarrow \infty,\; \forall j \in \{1,\ldots,n\}.$$

From this and from the standard Poincar\'{e} inequality proof we obtain
$$\|(v_k)_j\|_{0,2,\Omega} \rightarrow 0, \text{ as } k \rightarrow \infty,\; \forall j \in \{1,\ldots,n\}.$$
Thus, also similarly as in the proof of the last theorem,
  we may infer that $$\|v_k\|_{1,2,\Omega} \rightarrow 0, \text{ as } k \rightarrow \infty,$$
which contradicts $$\|v_k\|_{1,2,\Omega}=1,\; \forall k \in \mathbb{N}.$$

The proof is complete.
\end{proof}
\section{Conclusion} In this article we have developed a new proof for the Korn inequality in a specific n-dimensional context.

In a future research we intend to address more general models, including the corresponding results for manifolds in $\mathbb{R}^n$.


\begin{thebibliography}{}
%
%
\bibitem{1}
R.A. Adams and J.F. Fournier, {Sobolev Spaces}, 2nd edn.
 (Elsevier, New York, 2003).
 \bibitem{120}
F.S. Botelho, {Functional Analysis and Applied Optimization in Banach Spaces},
 (Springer Switzerland, 2014).
\bibitem{512} F.S.Botelho, {Functional Analysis, Calculus of Variations and Numerical Methods for Models in Physics and Engineering}, CRC Taylor and Francis, 2020.
\bibitem{30} P.G Ciarlet, Mathematical Elasticity, Volume II, Theory of Plates, North Holland Elsevier, 1997.
\end{thebibliography}
\end{document}